\documentclass{amsart}
\usepackage{amsmath}
\usepackage{color}
\usepackage{mathtools}

\newtheorem{theorem}{Theorem}[section]
\newtheorem{lemma}[theorem]{Lemma}
\newtheorem{proposition}[theorem]{Proposition}

 \theoremstyle{definition}
\newtheorem{definition}[theorem]{Definition}
\newtheorem{example}[theorem]{Example}

\theoremstyle{remark}

\numberwithin{equation}{section}

\begin{document}

\title[On solitons of the Hesse-Koszul flow]
{On solitons of a geometric flow on Hessian manifolds}

\author{Hanzhang Yin}
\address{School of Mathematics, Harbin Institute of Technology,
         Harbin, Heilongjiang 150001, China}
\email{YinHZ@hit.edu.cn}
\thanks{The author is supported by the Research Start-up Fund of Harbin Institute of Technology.}



\begin{abstract}

In this work, we introduce the Koszul solitons of the Hesse-Koszul flow, i.e., self-similar solutions for a geometric flow on Hessian manifolds. We obtain some partial differential equations and basic properties of Koszul solitons, and provide some examples. The key idea is to investigate the relationship between affine manifolds and complex manifolds, this method can also be used to study some other geometric problems on Hessian manifolds. It is worth noting that Hessian manifolds, as a special case of affine manifolds, have wide applications in statistics and information geometry.


\end{abstract}

\maketitle

\section{Introduction}

An affine manifold $(M,\nabla)$ is a differentiable manifold endowed with a flat connection $\nabla$. A Riemannian metric $g$ on an affine manifold $(M,\nabla)$ is called Hessian if $g$ can be locally expressed by $g=\nabla d\varphi$. Let $(M, \nabla, g_0)$ be a Hessian manifold, i.e., an affine manifold admits a Hessian metric $g_0$, consider the following geometric flow
\begin{equation}
\label{e1.2}
\left\{ \begin{aligned}
   &\frac{\partial }{\partial t}{{g}}=2\beta (g)  \\
      &g(0)={{g}_{0}}  \\
\end{aligned} \right.
\end{equation}
on $M$, where the second Koszul form $\beta (g)$ is defined by
\begin{equation}
\beta_{ij}=\frac{1}{2}\frac{\partial^2 \log \det[g_{pq}]}{\partial x^i \partial x^j}
\end{equation}
locally with respect to an affine coordinate system.
Mirghafouri-Malek \cite{r12} first introduced \eqref{e1.2} and proved that the evolved metrics $g(t)$ along the flow remain Hessian, in the terminology of Puechmorel-T{\^o} \cite{r13} the flow \eqref{e1.2} is called the Hesse-Koszul flow.

The concept of Ricci solitons, i.e., self-similar solutions of the Ricci flow, was introduced by Hamilton \cite{Ha}, for more related research, see \cite{Zhu17,DRS17}. In this paper, motivated by Ricci solitons, we introduce the self-similar solutions of \eqref{e1.2}, namely the Koszul solitons (see Definition \ref{d5.1}), and study the relationship between Koszul solitons and Ricci solitons (see Theorem \ref{t5.3}).

The Hesse-Koszul flow can solve some geometric problems, see \cite{r13,JY25}. There are also some other important problems in affine differential geometry, such as the classification of affine hyperspheres \cite{Ca58,CY86}, Chern's conjecture \cite{Klingler17}. The Hessian metric, which is also called the K\"{a}hler affine metric by Cheng-Yau \cite{CY82}, is an important concept in statistics and information geometry \cite{Am16, AN00, AJLS17}, it also appears in the study of physics \cite{SYZ96}. More details on the geometry of Hessian manifolds can be found in \cite{r11}.

The second Koszul form $\beta$ of a Hessian manifold plays a similar role as Ricci forms in K\"{a}hler geometry, a related problem is the real Calabi conjecture and the existence of Hesse-Einstein metrics \cite{CY82,r13}, such metrics can also be regarded as Koszul solitons (see Theorem \ref{t5.2}). By exploring the relationships between affine manifolds and complex manifolds (see Lemma \ref{l3.1}, Lemma \ref{l3.4} and Theorem \ref{t3.6}), we obtain the uniqueness of complete Hesse-Einstein metrics with $\beta>0$ (see Theorem \ref{t3.10}), and prove a rigidity theorem for bounded affine functions on complete Hessian manifolds with $\beta\leq0$ (see Theorem \ref{t3.8}).

We also study some elliptic and parabolic equations related to Koszul solitons, and obtain a relationship between certain fully non-linear equations on affine manifolds and complex manifolds (see Theorem \ref{t4.1}). An important example of this type of equations is the Monge-Amp\`{e}re equation, the existence of certain Koszul solitons can be reduced to solving Monge-Amp\`{e}re equation \eqref{en4.3}.

The rest of this paper is organized as follows. In Section 2, we provide some preliminaries which may be used later.
In Section 3, we study some geometric properties of affine manifolds. In Section 4, we introduce the definition of Koszul solitons and discuss their fundamental properties. In Section 5, we provide some examples of Koszul solitons. In Section 6, we consider the partial differential equations related to Koszul solitons. In Section 6, we prove Theorem \ref{t3.8}.

\section{Preliminaries}

In this section, we collect some useful results for flat connection and Hessian manifolds. The materials can be found in \cite{r11}.

A connection $\nabla$ is said to be flat if the torsion tensor $T$ and the curvature tensor $R$ vanish identically.
\begin{proposition}\emph{(\cite{r11})}
\label{prop1}
\

\emph{(1)}
\begin{minipage}[t]{0.92\linewidth}
Suppose that $M$ admits a flat connection $\nabla$. Then there exist local coordinate systems on $M$ such that $\nabla_{\partial /\partial x^i} \partial /\partial x^j=0$. The changes between such coordinate systems are affine transformations.
\end{minipage}

\emph{(2)}
\begin{minipage}[t]{0.92\linewidth}
Conversely, if $M$ admits local coordinate systems such that the changes of the local coordinate systems are affine transformations, then there exists a flat connection $\nabla$ satisfying $\nabla_{\partial /\partial x^i} \partial /\partial x^j=0$ for all such local coordinate systems.
\end{minipage}
\end{proposition}

Given a flat connection $\nabla$, a local coordinate system $\{x^1,\ldots,x^n\}$ satisfying $\nabla_{\frac{\partial}{\partial x^i}} \frac{\partial}{\partial x^j}=0$ is called an affine coordinate system with respect to $\nabla$. $f$ is called an affine map between two affine manifolds if and only if it is an affine map in local affine coordinates.
Let $(M,\nabla,g)$ be a Hessian manifold and $g$ can be locally expressed by
\[g_{ij}=\frac{\partial^2 \varphi}{\partial x^i \partial x^j}\]
where $\{x^1,\ldots,x^n\}$ is an affine coordinate system with respect to $\nabla$.
\begin{proposition}\label{p2.3}\emph{(\cite{r11})}
Let $(M,\nabla)$ be an affine manifold and $g$ be a Riemannian metric on $M$. Then $g$ is a Hessian metric if and only if $\dfrac{\partial g_{ij}}{\partial x^k}=\dfrac{\partial g_{kj}}{\partial x^i}$ for all $i,j,k$.
\end{proposition}

\begin{definition}(\cite{r11})
\label{Kos}
Let $(M,\nabla,g)$ be a Hessian manifold, and denote the Levi-Civita connection of
the Riemannian manifold $(M,g)$ by $\hat{\nabla}$. Let $\gamma=\hat{\nabla}-\nabla$. The first Koszul form $\alpha$ (a closed 1-form) and the second Koszul form $\beta$ (a symmetric bilinear form) for $(\nabla,g)$ are defined by
\[\nabla_X v=\alpha(X)v\]
\[\beta=\nabla \alpha\]
It follows that
\[\alpha(X)={\rm Tr} {\gamma_X}\]
and
\[\alpha_i=\frac{1}{2}\frac{\partial \log \det[g_{pq}]}{\partial x^i}=\gamma_{\;ki}^k\]
\[\beta_{ij}=\frac{\partial \alpha_i}{\partial x^j}=\frac{1}{2}\frac{\partial \log \det[g_{pq}]}{\partial x^i \partial x^j}\]
locally.
\end{definition}

\section{Some geometric properties}

In this section, we discuss the relationship between affine structures and complex structures, and derive some geometric properties of affine manifolds. These contents form the geometric foundation for the subsequent study of Koszul solitons.

Consider a flat manifold $(M,\nabla)$, and let $TM$ denote its tangent bundle with projection $\pi:TM\rightarrow M$. Given an affine coordinate system $\{x^1,\ldots,x^n\}$ on $M$, we define
\begin{equation}\label{e3.1}
z^j=\xi^j+\sqrt{-1}\xi^{n+j}
\end{equation}
where $\xi^i=x^i\circ \pi$ and $\xi^{n+j}=dx^i$. Since a affine bijection can be extended to a biholomorphic map, the $n$-tuples of functions given by $\{z^1,\ldots,z^n\}$ yields a holomorphic coordinate system on $TM$. Denote the complex structure tensor of the complex manifold $TM$ by $J_\nabla$, then
\begin{equation}
d\pi(\frac{\partial}{\partial \xi^i})=\frac{\partial}{\partial x^i},\;\;\;\;d\pi(\frac{\partial}{\partial \xi^{n+i}})=0,
\end{equation}
\begin{equation}\label{e3.3}
d\pi(\frac{\partial}{\partial z^i})=d\pi(\frac{\partial}{\partial \bar{z}^i})=\frac{1}{2}d\pi(\frac{\partial}{\partial \xi^i})=\frac{1}{2}\frac{\partial}{\partial x^i},
\end{equation}
for any $i=1,\ldots,n$, where $d\pi$ is the differential of $\pi$.
For a Riemannian metric $g$ on $M$ we set
\begin{equation}\label{e3.4}
g^T=\sum_{i,j=1}^{n}(g_{ij}\circ \pi)dz^id\overline{z}^j.
\end{equation}
Then $g^T$ is a Hermitian metric on the complex manifold $(TM,J_\nabla)$.
\begin{lemma}\label{l3.1}
Let $f$ be a smooth function on a flat manifold $(M,\nabla)$. For any multi-indices $\alpha,\beta$, we have
\begin{equation}
\frac{\partial^{|\alpha|+|\beta|}(f\circ\pi)}{\partial z^\alpha\partial\bar{z}^\beta}=\frac{1}{2^{|\alpha|+|\beta|}}\frac{\partial^{|\alpha|+|\beta|}f}{\partial x^\alpha\partial x^\beta}\circ\pi,
\end{equation}
where $f\circ\pi$ is a smooth function on $TM$, the affine coordinate system $\{x^1,\ldots,x^n\}$ and holomorphic coordinate system $\{z^1,\ldots,z^n\}$ are as in \eqref{e3.1}.
\end{lemma}
\begin{proof}
By the definition of the differential of a smooth map and equation \eqref{e3.3}, we have
\begin{equation}
\frac{\partial}{\partial z^i}(f\circ\pi)=\Big(d\pi(\frac{\partial}{\partial z^i})f\Big)\circ\pi=\frac{1}{2}\frac{\partial f}{\partial x^i}\circ\pi,
\end{equation}
\begin{equation}
\frac{\partial}{\partial \bar{z}^i}(f\circ\pi)=\Big(d\pi(\frac{\partial}{\partial \bar{z}^i})f\Big)\circ\pi=\frac{1}{2}\frac{\partial f}{\partial x^i}\circ\pi.
\end{equation}
Applying a new differential operator to both sides of the above two equations and repeating this process, we can prove the lemma.
\end{proof}
The following two properties show the relationship between Hessian manifolds and K{\"a}hler manifolds.
\begin{proposition}\label{p3.2} \emph{(\cite{r11})}
Let $(M,\nabla)$ be a flat manifold and $g$ a Riemannian metric on $M$. Then the following conditions are equivalent.

\noindent\emph{(1)} $g$ is a Hessian metric on $(M,\nabla)$.

\noindent\emph{(2)} $g^T$ is a K{\"a}hler metric on $(TM,J_\nabla)$.

\end{proposition}
Furthermore, we have
\begin{proposition}\label{p3.3} \emph{(\cite{r11})}
Let $(M, \nabla, g)$ be a Hessian manifold and $R_{i\bar{j}}^T$ be the Ricci tensor of the K{\"a}hler manifold $(TM,J_\nabla,g^T)$. Then we have
\[R_{i\bar{j}}^T=-\frac{1}{2}\beta_{ij}\circ \pi.\]
\end{proposition}
\begin{lemma}\label{l3.4}
Suppose $(M,\nabla ,g)$ is a complete connected affine Riemannian manifold. Then its tangent bundle $(TM,J_\nabla,g^T)$ is a complete Hermitian manifold.
\end{lemma}
\begin{proof}
We can use the same method as in Lemma 6.1 of Jiao-Yin \cite{JY25} to prove this lemma.
\end{proof}
Motivated by the bounded geometry of complex manifold (see Definition 1.1 of \cite{CY80} and Definition 2.1 of \cite{LT20}), we introduce the concept of bounded geometry for affine Riemannian manifolds:
\begin{definition}\label{db2.1}
Let $M$ be a complete affine Riemannian manifold. We say that $M$ has bounded geometry of order $l+\alpha$ if and only if $M$ admits a covering of affine coordinate charts $\{(V,(x^1,\cdots,x^n))\}$ and positive numbers $R,c$ such that:

\hspace*{0.2cm}(i)
\begin{minipage}[t]{0.91\linewidth}
for any $x_0\in M$ there is a coordinate chart $(V,(x^1,\cdots,x^n))$ with $x_0\in V$ and that, with respect to the Euclidean distance $d$ defined by $v^i$-coordinates, $d(x_0,\partial V)>R$;
\end{minipage}

\hspace*{0.1cm}(ii)
\begin{minipage}[t]{0.91\linewidth}
if $(g_{ij})$ denote the metric tensor with respect to $(V,(x^1,\cdots,x^n))$, then the components $g_{ij}$ of $g$ are uniformly bounded in the standard $C^{l+\alpha}$ norm in $(V,(v^1,\cdots,v^n))$ independent of $V$ and $(\delta_{ij})/c<(g_{ij})<c(\delta_{ij})$.
\end{minipage}
\end{definition}
The above definition is useful in Schauder-type estimates for partial differential equations on affine manifolds.
\begin{theorem}\label{t3.6}
Let $(M,\nabla ,g)$ be an affine Riemannian manifold with bounded geometry of order $l+\alpha$, then its tangent bundle $(TM,J_\nabla,g^T)$ has bounded geometry of order $l+\alpha$.
\end{theorem}
\begin{proof}
By Lemma \ref{l3.4}, $(TM,J_\nabla,g^T)$ is complete. Let $\{(V,(x^1,\cdots,x^n))\}$ be a a covering of affine coordinate charts on $M$. We define a covering of holomorphic coordinate charts on $TM$ by $\{(TV,(z^1,\cdots,z^n))\}$, where $TV$ is the tangent bundle of $V$, and $z^j$ is defined as in \eqref{e3.1}. For any $p\in TM$, there is a coordinate chart $(V,(x^1,\cdots,x^n))$ on $M$ with $\pi(p)\in V$ and $d(\pi(p),\partial V)>R$, then $(TV,(z^1,\cdots,z^n))$ is a holomorphic coordinate chart with $p\in TV$ and $d(p,\partial TV)>R$. By equation \eqref{e3.4} and Lemma \ref{l3.1}, we can prove that the components $g^T_{i\bar{j}}$ of $g^T$ are uniformly bounded in the standard $C^{l+\alpha}$ norm in $(TV,(z^1,\cdots,z^n))$ independent of $TV$ and $(\delta_{ij})/c<g^T_{i\bar{j}}<c(\delta_{ij})$.
\end{proof}

\section{Koszul solitons and Hesse-Einstein metrics}

In this section, we provide the definition of Koszul solitons and investigate their fundamental properties.
\begin{definition}\label{d5.1}
Let $(M,g(t))$ be a solution of the Hesse-Koszul flow \eqref{e1.2}, suppose that $\varphi_t:M\rightarrow M$ is a time-dependent family of affine bijections with $\varphi_0=id$, and $\sigma(t)$ is a time-dependent scale function with $\sigma(0)=1$, if
\begin{equation}\label{e5.1}
g(t)=\sigma(t)\varphi^*_tg(0),
\end{equation}
where $\varphi^*_t$ is the pull back map, then $(M,g(t))$ is called a Koszul soliton.
\end{definition}

Suppose that $(M,g(t))$ is a Koszul soliton, applying $\frac{\partial}{\partial t}$ to equation \eqref{e5.1} and
evaluating at $t=0$, we have
\begin{equation}
\frac{\partial}{\partial t}g(t)=\frac{\partial \sigma(t)}{\partial t}\varphi^*_tg(0)+\sigma(t)\frac{\partial}{\partial t}\varphi^*_tg(0),
\end{equation}
\begin{equation}\label{e5.3}
2\beta(g(0))=\sigma'(0)g(0)+\mathcal{L}_Vg(0),
\end{equation}
where $V=d\varphi_t/dt$ and $\mathcal{L}$ denotes the Lie derivative, the definition and relevant propositions of the Lie derivative can be found in \cite{CLN06}. By writing \eqref{e5.3} in local affine coordinates, we obtain
\begin{equation}\label{e5.4}
2\beta_{ij}=\sigma'(0)g_{ij}+\hat{\nabla}_i V_j+\hat{\nabla}_j V_i,
\end{equation}
where $\hat{\nabla}$ is the Levi-Civita connection of $g(0)$.

Next, we consider the relationship between Koszul solitons and Hesse-Einstein metrics. For a Hessian metric $g$, if $\beta(g)=\lambda g$ for some constant $\lambda$, then $g$ is called a Hesse-Einstein metric. If we set $V=0$ in \eqref{e5.4}, then $g(0)$ is a Hesse-Einstein metric.
\begin{theorem}\label{t5.2}
Let $M$ be an affine manifold with a Hesse-Einstein metric $g_0$, then there exists a Koszul soliton $(M,g(t))$ satisfying $g(0)=g_0$.
\end{theorem}
\begin{proof}
By the definition of Hesse-Einstein metric, we have $\beta(g_0)=\lambda g_0$ for some constant $\lambda$. Define $g(t):=(1+2\lambda t)g_0$, note that if $g_0$ is multiplied by a positive constant, the corresponding second Koszul form is invariant, then we have
\begin{equation}
\frac{\partial}{\partial t}g(t)=2\lambda g_0=2\beta(g_0)=2\beta(g(t))
\end{equation}
and $g(0)=g_0$.
\end{proof}
According to the above discussion, by associating the initial metric $g_0$ and the soliton $g(t)$, Hesse-Einstein metrics can be regarded as a special type of Koszul soliton. Let $g_0$ be the Hesse-Einstein metric as in the proof of Theorem \ref{t5.2} ($\beta(g)=\lambda g$), the soliton $g_0$ is called shrinking if $\lambda<0$, static if $\lambda=0$, and expanding if $\lambda>0$.
\begin{theorem}\label{t5.3}
Let $(M,g(t))$ be a Koszul soliton, then $(TM,g^T(t/4))$ be a K{\"a}hler-Ricci soliton.
\end{theorem}
\begin{proof}
By \eqref{e1.2} and Proposition \ref{p3.3}, we have
\begin{equation}
\left\{ \begin{aligned}
   &\frac{\partial }{\partial t}{g_{i\overline{j}}^T}(t)=-4R^T_{i\overline{j}} (g(t))  \\
      &g^T(0)= g_0^T.
\end{aligned} \right.
\end{equation}
By a parameter change, we get
\begin{equation}\label{e5.7}
\frac{\partial }{\partial t}{g_{i\overline{j}}^T}(t/4)=-R^T_{i\overline{j}} (g(t/4)).
\end{equation}
Let $\varphi_t$ be a family of affine bijections from $M$ to itself as in Definition \ref{d5.1}, then $\varphi_t$ can be extended to a biholomorphic map $\varphi^T_t$ from $TM$ to itself. By \eqref{e5.1} and \eqref{e3.4}, we have
\begin{equation}\label{e5.8}
g^T(t/4)=\sigma(t/4)(\varphi^T_{t/4})^*g^T(0).
\end{equation}
By \eqref{e5.7} and \eqref{e5.8}, we can prove this theorem.
\end{proof}
The above theorem gives the relationship between Koszul solitons and Ricci solitons. It is well-known that the complete K{\"a}hler Einstein metric (which can be regarded as a K{\"a}hler-Ricci soliton) with negative scalar curvature is unique up to scaling, we want to extend this conclusion to Hessian manifolds.
\begin{theorem}\label{t3.9}
Suppose $(M,g)$ and $(M',g')$ are two complete Hessian manifolds such that their second Koszul forms are equal to $K$ times their respective Hessian metrics, where $K$ is a positive constant. Then any affine bijection is an isometry.
\end{theorem}
\begin{proof}
By Proposition \ref{p3.2} and Lemma \ref{l3.4}, the tangent bundle $(TM,g^T)$ and $(TM',g^T)$ are two complete K{\"a}hler manifolds. By Proposition \ref{p3.3}, the Ricci forms of $(TM,g^T)$ and $(TM',g^T)$ are equal to $-\frac{K}{2}$ times their respective K{\"a}hler forms. Let $f$ be a affine bijection from $(M,g)$ to $(M',g')$, then $f$ can be extended to a biholomorphic map $f^T$ from $(TM,g^T)$ to $(TM',g^T)$. By Proposition 5.5 of \cite{CY80}, we obtain that $f^T$ is an isometry, then $f$ is an isometry.
\end{proof}
\begin{theorem}\label{t3.10}
The complete Hesse-Einstein metric with positive second Koszul form are unique up to scaling.
\end{theorem}
\begin{proof}
Let $g$ and $g'$ be two complete Hessian metrics on $M$ such that
\begin{equation}
\beta(g)=Kg,\;\;\;\;\beta(g')=K'g'
\end{equation}
for two positive constants $K$ and $K'$. Note that if $g'$ is multiplied by a positive constant, the corresponding second Koszul form is invariant, then
\begin{equation}
\beta\Big(\frac{K'}{K}g'\Big)=K'g'=K\Big(\frac{K'}{K}g'\Big).
\end{equation}
Consider the identity map from $M$ to itself, then by Theorem \ref{t3.9}, we have
\begin{equation}
g=\frac{K'}{K}g'.
\end{equation}
\end{proof}

\section{Examples for Koszul solitons}

By Theorem \ref{t5.2}, we obtain that Hesse-Einstein metrics can generate Koszul solitons. In this section, we give some examples of this type of solitons.

The following example is a shrinking Koszul soliton.
\begin{example}
Let $(\mathbb{R}^n,g=\nabla d\varphi)$ be a Hessian manifold, where
\[\varphi=\log\Big(1+\sum_{i=1}^n e^{x^i}\Big).\]
Hence
\[\det[g_{ij}]=\frac{e^{x^1}\cdots e^{x^n}}{f^{n+1}},\;\;\;{\rm where}\;f=1+\sum_{i=1}^n e^{x^i}.\]
Then we have
\[\alpha=\frac{1}{2}\Big(\sum_{i=1}^n dx^i-(n+1)d\log f\Big),\]
\[\beta=\nabla \alpha=-\frac{n+1}{2}\nabla d\log f=-\frac{n+1}{2}g.\]
The above Hesse-Einstein metric can be found in \cite{r11}. Finally, we can define a shrinking Koszul soliton $g(t):=(1-(n+1)t)g$.
\end{example}
The following two examples \ref{ex6.2} and \ref{ex6.3} are static Koszul solitons.
\begin{example}\label{ex6.2}
Let $(\mathbb{R}^n,g=\nabla d\varphi)$ be a Hessian manifold, where
\[\varphi=\sum_{i=1}^n(x^i)^2.\]
Hence $\det[g_{ij}]=1$ and $\beta=0$, we can define a static Koszul soliton $g(t):=g$.
\end{example}
\begin{example}\label{ex6.3}
Let $M$ be a compact Hessian manifold. Suppose that $M$ carries a parallel volume. Then $M$ admits a Riemannian flat metric $g$ (see \cite{CY82}). Hence $\beta(g)=0$ and we can define a static Koszul soliton $g(t):=g$.
\end{example}
The following two examples \ref{ex6.4} and \ref{ex6.5} are expanding Koszul solitons.
\begin{example}\label{ex6.4}
Let $(\Omega,g=\nabla d\varphi)$ be a Hessian manifold, where
\[\Omega=\Big\{x\in\mathbb{R}^n|x^n>\frac{1}{2}\sum_{i=1}^{n-1}(x^i)^2\Big\}\]
and
\[\varphi=-\log\Big(x^n-\frac{1}{2}\sum_{i=1}^{n-1}(x^i)^2\Big).\]
Hence
\[\det[g_{ij}]=f^{-n-1},\;\;\;{\rm where}\;f=x^n-\frac{1}{2}\sum_{i=1}^{n-1}(x^i)^2.\]
Then
\[\alpha=-\frac{n+1}{2}d\log f=\frac{n+1}{2}d\varphi\]
and
\[\beta=\frac{n+1}{2}\nabla d\varphi=\frac{n+1}{2}g.\]
The above Hesse-Einstein metric can be found in \cite{r11}. Finally, we can define a expanding Koszul soliton $g(t):=(1+(n+1)t)g$.
\end{example}

\begin{example}\label{ex6.5}
Let $(M,g)$ be a compact Hessian manifold. Suppose $\beta(g)$ is positive definite. Then there exists a Hessian metric $\tilde{g}$ on $M$ such that $\beta(\tilde{g})=\tilde{g}$ (see \cite{CY82}). Hence we can define a expanding Koszul soliton $\tilde{g}(t):=(1+2t)\tilde{g}$.
\end{example}

\section{PDEs related to Koszul solitons}

In this section, we consider the partial differential equations related to Koszul solitons.

Let $(M,g(t))$ be a Koszul soliton, then $g(t)$ satisfies \eqref{e1.2} and $g(0)$ satisfies \eqref{e5.3}. The Hesse-Koszul flow \eqref{e1.2} can be rewritten as a parabolic Monge-Amp\`{e}re equation:
\begin{equation}
\label{en4.1}
\left\{ \begin{aligned}
   &\frac{\partial u }{\partial t}=\log \frac{\det ({{g}_{0}}+2t\beta ({{g}_{0}})+\nabla du )}{\det ({{g}_{0}})},  \\
   &{{g}_{0}}-t\beta ({{g}_{0}})+\nabla du >0,  \\
   &u (0)=0.  \\
\end{aligned} \right.
\end{equation}
where $u$ is an unknown scalar function. Indeed, if $g(t)$ is a solution to \eqref{e1.2}, we define $u(x,t)$ by
\[u(x,t)=\int_{0}^{t}{\log \frac{\det g(x,s)}{\det {{g}_{0}}(x)}}ds.\]
Then
\[\frac{\partial u}{\partial t}=\log \frac{\det g(x,t)}{\det {{g}_{0}}(x)},\ \ u(x,0)=0\]
By direct calculations, we have
\[\frac{\partial}{\partial t}(g(t)-g_0-2t\beta(g_0)-\nabla du)=0\]
and
\[(g(t)-g_0-2t\beta(g_0)-\nabla du)|_{t=0}=0.\]
Hence $g(t)=g_0+2t\beta(g_0)+\nabla d\varphi>0$ and $u$ satisfies \eqref{en4.1}.

Conversely, given $u(t)$ solving \eqref{en4.1}, we have $g(t):={g}_{0}+2t\beta ({{g}_{0}})+\nabla du$ satisfies
\[\frac{\partial }{\partial t}g(t)=2\beta ({{g}_{0}})+\nabla d\log \frac{\det (g(t))}{\det ({{g}_{0}})}=2\beta ({{g}_{t}})\]
and $g(0)=g_0$. It follows that $g(t)$ is a solution to \eqref{e1.2}.

Next, we consider the equation \eqref{e5.3} with $V=0$ and $\sigma'(0)=K$, where $K$ is a constant. Suppose that there exists a function $\hat{h}$ defined on $M$ such that
\begin{equation}\label{en4.2}
2\beta(g(0))=Kg(0)+\nabla d\hat{h}.
\end{equation}
Then if $u$ is a function defined on $M$ and satisfies the following Monge-Amp\`{e}re equation:
\begin{equation}\label{en4.3}
\det(g(0)_{ij}+u_{ij})=e^{Ku-\hat{h}}\det(g(0)_{ij}),
\end{equation}
we have
\begin{equation}\label{en4.4}
\begin{aligned}
\beta_{ij}(g(0)+\nabla du)&=\frac{1}{2}\frac{\partial^2}{\partial x^i\partial x^j}[\log \det(g(0)_{pq}+u_{pq})]\\
&=\frac{1}{2}\frac{\partial^2}{\partial x^i\partial x^j}[Ku-\hat{h}+\log \det(g(0))]\\
&=\frac{K}{2}(g(0)+\nabla du).
\end{aligned}
\end{equation}
It follows that $g(0)+\nabla du$ is a Hesse-Einstein metric, and $g(t)=(1+Kt)(g(0)+\nabla du)$ is a Koszul soliton.

\textbf{Remark}: Let $(M,g)$ be a Hessian manifold with $\beta(g)>0$, then
\begin{equation}
2\beta(\beta(g))=2\beta(g)+\nabla d\log\frac{\det(\beta(g))}{\det g}.
\end{equation}
Set $g(0)=\beta(g)$, we conclude that $g(0)$ satisfies \eqref{en4.2}. Similarly, if $\beta(g)<0$, we can take $g(0)=-\beta(g)$.

Next, we generalize equations \eqref{en4.1} and \eqref{en4.3} to general fully nonlinear equations. Let $(M, \nabla, g)$ be an affine Riemannian manifold of dimension $n$. Given a $(0,2)$-tensor $\chi$ on $(M,g)$, for function $u\in C^2(M,\mathbb{R})$ or $u\in C^2(M\times \mathbb{R},\mathbb{R})$, we have a new $(0,2)$-tensor $\alpha=\chi+\nabla du$, and we can define $A_j^i=g^{ip}\alpha_{jp}$. Consider the equations for $u$ as follows:
\begin{equation}\label{e4.1}
F(A)=h(x,u),\;\;\;\;{\rm if}\;u\in C^2(M,\mathbb{R})
\end{equation}
\begin{equation}\label{e4.2}
F(A)-u_t=h(x,u),\;\;\;\;{\rm if}\;u\in C^2(M\times \mathbb{R},\mathbb{R})
\end{equation}
for a given function $h$ on $M$, where
\begin{equation}
F(A)=f(\lambda_1,\ldots,\lambda_n)
\end{equation}
is a smooth symmetric function of the eigenvalues of $A$ and $f_i>0$ for all $i$. Equations \eqref{e4.1} and \eqref{e4.2} are well-defined, i.e., independent of the choice of affine coordinate system, where \eqref{e4.1} is an elliptic equation and \eqref{e4.2} is a parabolic equation. By taking $\chi=g(0)$, $f(\lambda_1,\ldots,\lambda_n)=\log\lambda_1\ldots\lambda_n$ and $h(x,u)=Ku-\hat{h}$, we find that \eqref{en4.3} is a special case of \eqref{e4.1}. By taking $\chi=g(0)+2t\beta ({{g}_{0}})$, $f(\lambda_1,\ldots,\lambda_n)=\log\lambda_1\ldots\lambda_n$ and $h(x,u)=0$, we find that \eqref{en4.1} is a special case of \eqref{e4.2}.
\begin{theorem}\label{t4.1}
Let $(M, \nabla, g)$ be an affine Riemannian manifold of dimension $n$. If $u\in C^2(M,\mathbb{R})$ satisfies \eqref{e4.1}, then $u\circ\pi$ satisfies
\begin{equation}
F((g^T)^{i\bar{p}}(\chi^T_{j\bar{p}}+4(u\circ\pi)_{j\bar{p}}))=h(\pi(x),u\circ\pi)
\end{equation}
on $TM$. If $u\in C^2(M\times \mathbb{R},\mathbb{R})$ satisfies \eqref{e4.2}, then $u\circ\pi:=u(\pi(\cdot),\cdot)$ satisfies
\begin{equation}\label{e4.5}
F((g^T)^{i\bar{p}}(\chi^T_{j\bar{p}}+4(u\circ\pi)_{j\bar{p}}))-(u\circ\pi)_t=h(\pi(x),u\circ\pi)
\end{equation}
on $TM\times \mathbb{R}$, where $\chi^T$ is defined similarly to $g^T$.
\end{theorem}
\begin{proof}
We take the proof of \eqref{e4.5} as an example. By Lemma \ref{l3.1} and equation \eqref{e4.2}, we have
\begin{equation}
\begin{aligned}
&F((g^T)^{i\bar{p}}(\chi^T_{j\bar{p}}+4(u\circ\pi)_{j\bar{p}}))-(u\circ\pi)_t\\
=&F((g^{ip}\circ\pi)(\chi_{jp}\circ\pi+u_{jp}\circ\pi))-(u\circ\pi)_t\\
=&F\circ\pi-u_t\circ\pi\\
=&h\circ\pi\\
=&h(\pi(x),u\circ\pi)
\end{aligned}
\end{equation}
\end{proof}
The above theorem is useful for studying elliptic and parabolic equations on affine manifolds. For example, we can derive the a priori estimates and some properties of solutions for equations on affine manifolds from the equations on complex manifolds.
\begin{example}
For Hessian manifold $(M, \nabla, g)$, let $\chi=g$, $f(\lambda_1,\ldots,\lambda_n)=\log\lambda_1\ldots\lambda_n$ and $h(x,u)=\hat{h}(x)+u$, where $\hat{h}(x)$ is a smooth function on $M$, then \eqref{e4.1} can be written as:
\begin{equation}\label{e4.7}
\det(g_{ij}+u_{ij})=e^{\hat{h}+u}\det(g_{ij}).
\end{equation}
Such equations are called the real Monge-Amp\`{e}re equations, they are useful in affine differential geometry, see \cite{CY82}, and this equation is related to a class of Koszul solitons, see \eqref{en4.3} and \eqref{en4.4}. If \eqref{e4.7} admits a solution, then by Theorem \ref{t4.1}, the complex Monge-Amp\`{e}re equation
\begin{equation}\label{e4.8}
\det(g^T_{i\bar{j}}+4v_{i\bar{j}})=e^{\hat{h}\circ\pi+v}\det(g^T_{i\bar{j}})
\end{equation}
has a solution $v=u\circ\pi$ on $TM$. Assume that $M$ has bounded geometry, $\hat{h}$ and $u$ belong to appropriate Banach spaces (see \cite{CY80}), then by Theorem \ref{t3.6}, we can establish the a priori estimates for \eqref{e4.7} from the estimates for \eqref{e4.8} given in \cite{CY80}.
\end{example}

\section{A rigidity theorem for affine functions}

In the previous sections, we use the relationship between affine manifolds and complex manifolds to study Koszul solitons. In this section, we present some other geometric applications of these methods on Hessian manifolds. It is obvious that Bounded affine functions on $\mathbb{R}^n$ are constant. In the following, we attempt to generalize this to Hessian manifolds.
\begin{theorem}\label{t3.7}
Let $(M,g_M)$ be a complete Hessian manifold with the second Koszul form bounded from above by $K_1$ (i.e. $\beta(g_M)\leq K_1g_M$). Let $(N,g_N)$ be another affine Riemannian manifold with the holomorphic bisectional curvature (see \cite{Yau78}) of its tangent bundle $(TN,g_N^T)$ bounded from above by a negative constant $K_2$. Then if there is a non-constant affine map $f$ from $M$ to $N$, we have $K_1\geq 0$ and the pull back metric $f^*g_N$ satisfies:
\begin{equation}\label{e3.8}
f^*g_N\leq -\frac{K_1}{2K_2}g_M.
\end{equation}
In particular, if $K_1\leq 0$, every affine map from $M$ into $N$ is constant.
\end{theorem}
\begin{proof}
Let $(TM,g^T_M)$ be the tangent bundle of $(M,g_M)$, by Proposition \ref{p3.2} and Lemma \ref{l3.4}, $(TM,g^T_M)$ is a complete K{\"a}hler manifold. By Proposition \ref{p3.3}, we have
\[{\rm Ric}(g^T_M)\geq-\frac{K_1}{2}.\]
The affine map $f$ can be extended to a holomorphic map $f^T$ from $(TM,g^T_M)$ to $(TN,g_N^T)$, by Theorem 2 of Yau \cite{Yau78}, we have $K_1\geq 0$ and
\begin{equation}\label{e3.9}
(f^T)^*g^T_N\leq -\frac{K_1}{2K_2}g^T_M.
\end{equation}
Then \eqref{e3.8} follows immediately from \eqref{e3.9} and \eqref{e3.4}.
\end{proof}
\begin{theorem}\label{t3.8}
Let $M$ be a complete Hessian manifold with non-positive second Koszul form. Then every bounded affine function on $M$ is constant.
\end{theorem}
\begin{proof}
Let $f$ be a bounded affine function on $M$ such that $-c<f<c$ for some constant $c>0$. Note that $(-c,c)$ is an affine manifold, we define a Riemannian metric $g(x)=x^2+4c^2$ on $(-c,c)$ for any $x\in (-c,c)$. Consider the tangent bundle $(T(-c,c),g^T)$ of $((-c,c),g)$, by Lemma \ref{l3.1}, we have
\begin{equation}
\begin{aligned}
R_{1\bar{1}1\bar{1}}&=-\frac{\partial^2g^T}{\partial z^2}+(g^T)^{-1}\frac{\partial g^T}{\partial z}\frac{\partial g^T}{\partial \bar{z}}\\
&=-\frac{1}{4}\frac{\partial^2g}{\partial x^2}\circ\pi+\frac{1}{4}g^{-1}\Big(\frac{\partial g}{\partial x}\Big)^2\circ\pi,
\end{aligned}
\end{equation}
where $R_{1\bar{1}1\bar{1}}$ is the curvature tensor of $(T(-c,c),g^T)$. Denote the holomorphic bisectional curvature of $(T(-c,c),g^T)$ by ${\rm BK}$, then
\begin{equation}
{\rm BK}=-\frac{1}{4}g^{-2}\frac{\partial^2g}{\partial x^2}\circ\pi+\frac{1}{4}g^{-3}\Big(\frac{\partial g}{\partial x}\Big)^2\circ\pi<c'<0.
\end{equation}
Finally, by Theorem \ref{t3.7}, we obtain that the affine map $f$ from $M$ into $(-c,c)$ is constant.
\end{proof}

\textbf{Acknowledgement.} I would like to thank Professor Jiao Heming for suggesting him to study some elliptic/parabolic equations on Hessian manifolds.

\end{document}